\documentclass[12pt]{amsart}
\usepackage{amssymb,amsfonts,amsmath,amsopn,amstext,amscd,latexsym,amsthm,enumerate,mathrsfs,bbm,color}

\usepackage[margin=36mm]{geometry}
\usepackage[all,cmtip]{xy}
\usepackage{longtable}

\newtheorem{theorem}{Theorem}[section]
\newtheorem{thm}{Theorem}
\newtheorem{lemma}[theorem]{Lemma}

\newtheorem{lem}[theorem]{Lemma}

\theoremstyle{remark}

\numberwithin{equation}{section}

\DeclareMathOperator{\Irr}{Irr}     \DeclareMathOperator{\IBr}{IBr}

\begin{document}

\title[Taketa's theorem]{Isaacs' Generalization of Taketa's theorem}

\author[X. Chen]{Xiaoyou Chen}
\address{School of Mathematics and Statistics, Henan University of Technology, Zhengzhou 450001, China}
\email{cxymathematics@hotmail.com}

\author[M. L. Lewis]{Mark L. Lewis}
\address{Department of Mathematical Sciences, Kent State University, Kent, OH 44242, USA}
\email{lewis@math.kent.edu}


\date{\today}
%\thanks{}


\begin{abstract}
We generalize the definition of pseudo monomial characters and $M$-groups to the Brauer character and Isaacs' $\pi$-partial character settings.  We prove an analogs of Isaacs's generalization of Taketa's theorem in those settings.   We consider other analogs of results regarding $M$-groups in those settings.
\end{abstract}
\maketitle
%%%%%%%%%%%%%%%%%%%%%%%%%%%%%%%%%%%%%%%%%%%%%%%%%%%%%%%%%%%%%%%%%%%%%%%%%%%%%%%%%%%%%%%%%%%%%%%%%%%%%%

\section{Introduction}

In this paper, all groups are finite.  The reader is referred to \cite{Isaacs1976} and \cite{Navarro1998} for notation.  Let $G$ be a group, and write ${\rm Irr} (G)$ for the set of irreducible (complex) characters of $G$.  Recall that a character $\chi$ of $G$ is {\it monomial} if $\chi$ is induced from a linear character of some subgroup of $G$.  If every $\chi\in {\rm Irr} (G)$ is monomial, then $G$ is said to be an {\it $M$-group}.  A well-known theorem of Taketa (Corollary 5.13 of \cite{Isaacs1976}) states that an $M$-group is necessarily solvable.

This paper is motivated by Theorem A of  \cite{Isaacs19842} where Isaacs generalizes Taketa's result to classes of groups.  Take $G$ to be a group and take $\chi$ to be an irreducible character of $G$.  Define the set of degrees:
$$D (\chi) = \{ \psi (1) \mid \psi \in {\rm Irr} (H) ~\text{for some}~H \leq G~\text{and}~\psi^{G} = \chi \}.$$
Following \cite{Gunter}, when ${\rm gcd} (D(\chi)) = 1$ for each character $\chi \in {\rm Irr}(G)$, we say $G$ is {\it a pseudo-$M$-group}.  We can see that pseudo-$M$-groups are generalizations of $M$-groups.  In Theorem C of \cite{Isaacs19842}, Isaacs uses his generalization of Taketa's theorem to prove that pseudo-$M$-groups are solvable.  In her dissertation \cite{Gunter}, Gunter gave this definition a name: pseudo-M-groups.

In this paper, we will prove an analog of Isaacs' generalization in the context of Brauer characters.  We also will  give an analog of Gunther's definition in the Brauer context and use the generalization of Taketa's result to prove solvability in this context.

Let $p$ be a prime.  % and $F$ be an algebraically closed field of characteristic $p$.
Denote by ${\rm IBr}(G)$ the set of irreducible ($p$-)Brauer characters of $G$.  A $p$-Brauer character of $G$ is monomial if it is induced from a linear $p$-Brauer character of some subgroup (not necessarily proper) of $G$.  This definition was introduced by Okuyama in \cite{Okuyama} using module theory.  A group $G$ is called an {\it $M_{p}$-group} if every irreducible $p$-Brauer character of $G$ is monomial.  In particular, Okuyama proves in \cite{Okuyama} that $M_{p}$-groups  are solvable.  Okuyama also remarked in \cite{Okuyama} that an $M$-group are necessarily an $M_{p}$-group for every prime $p$.  However an $M_{p}$-group need not be an $M$-group; for example, ${\rm SL}(2, 3)$ is an $M_{2}$-group, not an $M$-group.   In fact, a group may be an $M_p$-group for every prime $p$ and not be an $M$-group.  In \cite{moreto}, an example of a $\{p,q\}$-groups is given that is both an $M_p$- and and $M_q$-group, but is not an $M$-group.

We generalize Okuyama's theorem on the solvability of $M_{p}$-groups in this note by using classes of groups along the lines of Isaacs' Theorem A in \cite{Isaacs19842}.  Note that we are using the definition of the kernel of Brauer character found on page 39 of \cite{Navarro1998}.

\begin{thm}\label{theorem1}
Let $\mathcal{F}$ be a family of groups that is closed under isomorphism, subgroups and extensions, let $G$ be a group, and let $p$ be a prime. Assume that ${\bf O}_{p} (G) = 1$.  If for every Brauer character $\varphi \in {\rm IBr} (G)$, there exists a subgroup $H \leq G$ and a Brauer character $\psi \in {\rm IBr} (H)$ so that $\psi^{G} = \chi$ and $H/\ker\psi \in \mathcal{F}$, then $G \in \mathcal{F}$.
\end{thm}

Let $\mathcal{F} = \{\text{solvable groups}\}$.  We have from Theorem \ref{theorem1} that if $G$ is an $M_{p}$-group, then $G$ is solvable.  Motivated by Gunter in \cite{Gunter}, we define pseudo $M_p$-groups. For a Brauer character $\varphi \in {\rm IBr}(G)$, write
$$S (\varphi) = \{ \psi \in {\rm IBr} (H)\mid H \leq G, \psi^{G} = \varphi \}.$$
When ${\rm gcd} \{ \psi(1) \mid \psi \in S(\varphi) \} = 1$, we define $\varphi$ to be {\it pseudo monomial}.  We say $G$ is a {pseudo $M_p$-group} if every Brauer character in ${\rm IBr} (G)$ is pseudo-monomial.  With this definition, we have another generalization about the solvability of $M_{p}$-groups.  As we said above, this next theorem is motivated by Theorem C of \cite{Isaacs19842}.

\begin{thm} \label{pseudo M_p}
Let $G$ be a group and let $p$ be a prime.  Assume ${\bf O}_{p} (G) = 1$.
    %For each $\varphi\in {\rm IBr}(G)$, write
    %$$S(\varphi)=\{\psi\in {\rm IBr}(H)\mid H\subseteq G, \psi^{G}=\varphi\}.$$
    %Assume that for each $\varphi\in {\rm IBr}(G)$,
    %${\rm gcd}\{\psi(1)\mid \psi\in S(\varphi)\}=1$.
If $G$ is a pseudo $M_p$-group, then $G$ is solvable.
\end{thm}

%
%Following Okuyama \cite{Okuyama}, a group $G$ is said to be an {\it $M_{p}$-group} if every Brauer character $\varphi\in {\rm IBr} (G)$ is monomial. In \cite{Okuyama}, Okuyama proves that $M_{p}$-groups must be solvable.  Okuyama also remarks in \cite{Okuyama} that the Fong-Swan theorem implies that an $M$-group must be an $M_{p}$-group for every prime $p$; however an $M_{p}$-group may not be an $M$-group.  In fact, a group $G$ that is an $M_p$-group for every prime $p$ will not necessarily be an $M$-group.

%Similar to Gunter's definition, we define pseudo-$M_{p}$-groups.  Take $\varphi$ to be an irreducible Brauer character of $G$.  Write
%$$S(\varphi) = \{ \psi \in {\rm IBr} (H) \mid H \leq G, \psi^{G} = \varphi \}.$$
%When ${\rm gcd} \{ \psi(1) \mid \psi\ in S(\varphi) \} = 1$, we say that $\varphi$ is a {\it pseudo-monomial Brauer character}.  If all irreducible Brauer characters of $G$ are pseudo-monomial, then $G$ is said to be a {\it pseudo-$M_{p}$-group}.
%
In Gunter's dissertation \cite{Gunter}, she further explores pseudo-$M$-groups.  In particular, she proves in a number of situations that pseudo $M$-groups will actually be $M$-groups.  For the rest of this paper, we will further explore pseudo $M$-groups and pseudo $M_p$-groups.  We next prove that Okuyama's observation about $M$-groups has an analog to pseudo $M$-groups.

\begin{thm}\label{thm01}
If $G$ is a pseudo-$M$-group, then $G$ is a pseudo-$M_{p}$-group for every prime $p$.
\end{thm}

One important fact about $M$-groups is that monomiality is not inherited by normal subgroups (see \cite{dade} and \cite{vand}); however, in view of a theorem of Dornoff (see Problem 6.9 of \cite{Isaacs1976})) a normal Hall subgroup of an $M$-group is still an $M$-group.  The next theorem shows that there is an analog of Dornhoff's theorem for pseudo-$M_p$-groups.

\begin{thm}\label{thm02}
The following are true:
\begin{enumerate}
\item If $G$ is a pseudo-$M$-group and $N$ is a normal Hall subgroup of $G$, then $N$ is a pseudo-$M$-group.
\item If $G$ is a pseudo-$M_{p}$-group and $N$ is a normal Hall subgroup of $G$, then $N$ is also a pseudo-$M_{p}$-group.
\end{enumerate}
\end{thm}

%We review some of the key points of Isaacs' $\pi$-theory,
%and one can see \cite{Isaacs1984} and \cite{Isaacsb2018}.

In fact, we will also show that these results can also be shown in the context of Isaacs' $\pi$-partial characters.  Since the statements of the results are similar, we do not repeat them here.  We will review the details of Isaacs' $\pi$-theory and state the analogs of our results as it is appropriate.

\section{Isaacs generalization of Taketa's theorem}

We first introduce the definition of classes $\mathcal{F}$ of groups.  We say that $\mathcal{F}$ is a {\it class of groups}  if $\mathcal{F}$ satisfies the following conditions.

{\rm (i)} If $K \cong H \in \mathcal{F}$, then $K \in \mathcal{F}$. (Closure under isomorphisms)

{\rm (ii)} If $K\subseteq H\in \mathcal{F}$, then $K\in \mathcal{F}$. (Closure under subgroups)

{\rm (iii)} If $N\lhd G$ and $N\in \mathcal{F}$ and $G/N\in \mathcal{F}$, then $G\in \mathcal{F}$.  (Closure under extensions)

Two examples of classes are solvable groups and $\pi$-groups for a fixed set $\pi$ of primes.  We next show that the residual for a class of groups lies in the class.

\begin{lem}\label{lemma1}
If $\mathcal{F}$ is a class of groups and $G$ is any group, then $G$ has a unique $\mathcal{F}$-residual: i.e., there exists a normal subgroup $G^{\mathcal{F}}$ that is minimal with respect to the property that $G/G^{\mathcal{F}}\in \mathcal{F}$.
\end{lem}

\begin{proof}
Note that $G^{\mathcal{F}}/(G^{\mathcal{F}})^{\mathcal{F}}\in \mathcal{F}$ and $G/(G^{\mathcal{F}})^{\mathcal{F}}/G^{\mathcal{F}}/(G^{\mathcal{F}})^{\mathcal{F}}
\cong G/G^{\mathcal{F}}\in \mathcal{F}$.  We have by (iii) that $G/(G^{\mathcal{F}})^{\mathcal{F}}\in \mathcal{F}$.  It follows that $(G^{\mathcal{F}})^{\mathcal{F}}=G^{\mathcal{F}}$.
\end{proof}

We now show that under an additional hypothesis, we can prove Theorem \ref{theorem1} for a set of primes in the more general setting of Isaacs' $\pi$-theory.  Let $\pi$ be a set of primes and $G$ be a $\pi$-separable group (recall that $G$ is a $\pi$-separable group if $G$ has a normal series such that each factor is either a $\pi$-group or a $\pi'$-group).   Isaacs introduced $\pi$-theory in \cite{Isaacs1984}.   In his recent monograph \cite{solv text}, Isaacs gives a full account of $\pi$-theory.

The {\it $\pi$-partial characters} of $G$ are the restrictions of characters of $G$ to $G^o$, where $G^o$ denotes the set of $\pi$-elements of $G$.  Denote by ${\rm I}_{\pi} (G)$ the set of those $\pi$-partial characters that cannot be written as a sum of other $\pi$-partial characters.  The members of ${\rm I}_{\pi} (G)$ are called the {\it irreducible $\pi$-partial characters} of $G$.  Isaacs shows that ${\rm I}_{\pi} (G)$ plays the same role for $\pi$-partial characters as ${\rm IBr} (G)$ for $p$-Brauer characters, where $p$ is a prime and ${\rm IBr} (G)$ is the set of irreducible $p$-Brauer characters of $G$.  In fact, when $\pi = \{ p \}'$ and $G$ is $p$-solvable, ${\rm I}_{p'} (G)$ and ${\rm IBr} (G)$ coincide.
%We refer the reader to \cite{Isaacs1976} and \cite{Navarro1998}
%for notation and terminology of character theory and Brauer character theory, respectively.
%If $G$ is $\pi$-separable, then

Isaacs gives two different proofs that the set ${\rm I}_{\pi}(G)$ is a basis for the vector space of class functions on $G^o$, one in \cite{Isaacs1984} and a different proof in \cite{Isaacs1994}. Also, Isaacs defines a canonical subset ${\rm B}_{\pi} (G) \subset {\rm Irr} (G)$, where ${\rm Irr} (G)$ is the set of irreducible (complex) characters of $G$, and shows that the restriction map ${\rm B}_{\pi}(G)\rightarrow {\rm I}_{\pi}(G)$ defines a bijection between these sets, see \cite[Corollary 10.2]{Isaacs1984}.  In fact, we use the set ${\rm B}_\pi (G)$ to define the kernel of the $\pi$-partial characters.  Given $\phi \in {\rm I}_\pi (G)$, there is a unique character $\chi \in {\rm B}_\pi (G)$ so that $\chi^o = \phi$.  We define $\ker \phi = \ker \chi$.   In \cite{my pre}, we show that this definition for the kernel of partial characters is consistent with the definition for the kernel of Brauer characters.

In \cite{Isaacs1994}, Isaacs defines induction for $\pi$-partial characters and just like for ordinary characters, a $\pi$-partial character is monomial if it is induced from a linear $\pi$-partial character of some subgroup.  If every $\pi$-partial character $\varphi\in {\rm I}_{\pi}(G)$ is monomial, then we say $G$ is an {\it $M_{\pi}$-group}.  (Note that our notation is a little different from that of \cite{DIG}.)  As we stated above, when $\pi = p'$, we have ${\rm I}_{\pi} (G) = {\rm IBr}(G)$, and so, $M_{\pi} = M_{p'}$.  In particular, note that $M_p$ and $M_{\{ p \}}$ are not the same!

In \cite{DIG}, Dastouri, Iranmanesh and Ghasemi prove when $G$ is an $M_{\pi}$-group then $G/{\bf O}_{\pi'} (G)$ is solvable.
%By using the similar proof of Theorem \ref{theorem1},
We generalize their result to the following version.

\begin{theorem}\label{theorem2}
Let $\mathcal {F}$ be a class of groups, let $\pi$ be a set of primes, and let $G$ be a $\pi$-separable group so that ${\bf O}_{\pi'} (G) = 1$.  If for every $\pi$-partial character $\chi \in {\rm I}_{\pi} (G)$, there exists a subgroup $H \leq G$ and a $\pi$-partial character $\psi \in {\rm I}_{\pi} (H)$ such that $\psi^{G} = \chi$ and $H/\ker\psi \in \mathcal {F}$, then $G \in \mathcal{F}$.
\end{theorem}

We now prove Theorem \ref{theorem1} and Theorem \ref{theorem2} together.

\begin{proof}[Proof of Theorems \ref{theorem1} and \ref{theorem2}]
Set ${\rm I} (G)$ to be ${\rm IBr} (G)$ if we are are in Theorem \ref{theorem1} and ${\rm I}_{\pi} (G)$ if we are in Theorem \ref{theorem2}.  Write $N = G^{\mathcal{F}}$ and our goal is to prove that $N = 1$.  Once our goal is achieved, we will have $G = G/1 = G/N \in \mathcal{F}$;
    %Our object is to show that $N=1$ and so we assume the contrary.
thus we assume that $N > 1$, and we work to obtain a contradiction.

When we are in Theorem \ref{theorem1} take $O$ to be ${\bf O}_p (G)$ and when we are in Theorem \ref{theorem2} take $O$ to be ${\bf O}_{\pi'} (G)$.  Since $\cap_{\varphi \in {\rm I}(G)} \ker\varphi =  O = 1$ (see Lemma 2.1 of \cite{super} for Brauer characters and \cite{my pre} for partial characters), we see that $N \not\leq \ker\varphi$ for some $\varphi\in {\rm IBr}(G)$, and we can choose either a Brauer character or a $\pi$-partial $\varphi$ with minimum possible degree.  It follows by hypothesis that $\varphi = \psi^{G}$ where $\psi\in {\rm I} (H)$ and $H/\ker\psi \in \mathcal{F}$.

If $H = G$, then $\varphi = \psi$, and it follows that $N = G^{\mathcal{F}} \leq \ker\psi = \ker\varphi$, a contradiction.

Now, consider $\theta = (1_{H^o})^{G}$, where $H < G$ and $1_{H^o}$ is the principal Brauer character or $\pi$-partial character of $H$.  Then the principal Brauer character or $\pi$-partial character $1_{G^o}$ of $G$ is a constituent of $\theta$, and thus, every other irreducible constituent of $\theta$ has degree that is strictly less than $\theta (1) = |G: H| \leq \varphi (1) = |G: H| \psi (1)$.  We deduce, by the choice of $\varphi$, that $N$ is contained in the kernel of every irreducible constituent of $\theta$, and thus
$$N \leq \ker\theta = \ker((1_{H^o})^{G}) \leq H.$$

Now, $H/N \leq G/N \in \mathcal {F}$, and thus, $H/N\in \mathcal {F}$, and so, $H^{\mathcal {F}} \leq N$. But
$$N/H^{\mathcal {F}} \leq H/H^{\mathcal {F}} \in \mathcal {F}$$
and therefore, $N^{\mathcal {F}} \leq H^{\mathcal {F}}$.  Since $N^{\mathcal {F}}= N$ by Lemma \ref{lemma1}, we have  $H^{\mathcal {F}} = N$.

We know that $H/\ker\psi \in \mathcal {F}$, and so, $H^{\mathcal {F}} \leq \ker\psi$.   Thus, $N \subseteq \ker\psi$ and since $N \lhd G$, we conclude that $N \subseteq \ker (\psi^{G}) = \ker \varphi$.  This is a contradiction.
\end{proof}

\section{Pseudo $M_p$-groups}

We now work to prove the solvability of pseudo $M_p$-groups.  First, we need to define the analogs for $\pi$-partial characters.

We now define pseudo $M_\pi$-groups. For a $\pi$-partial character $\varphi\in {\rm I}_\pi (G)$, write
$$S(\varphi) = \{\psi\in {\rm I}_\pi (H) \mid H \leq G, \psi^{G} = \varphi \}.$$
When ${\rm gcd} \{\psi(1) \mid \psi \in S(\varphi) \} = 1$, we define $\varphi$ to be {\it pseudo monomial}.  We say $G$ is a {pseudo $M_\pi$-group} if every $\pi$-partial character in ${\rm I}_\pi (G)$ is pseudo-monomial.  With this definition we have another generalization about the solvability of $M_\pi$-groups.

\begin{theorem} \label{pseudo M_pi}
Let $G$ be a group and let $\pi$ be a set of primes.  Assume $G$ is $\pi$-separable and ${\bf O}_{\pi'} (G) = 1$.
    %For each $\varphi\in {\rm IBr}(G)$, write
    %$$S(\varphi)=\{\psi\in {\rm IBr}(H)\mid H\subseteq G, \psi^{G}=\varphi\}.$$
    %Assume that for each $\varphi\in {\rm IBr}(G)$,
    %${\rm gcd}\{\psi(1)\mid \psi\in S(\varphi)\}=1$.
If $G$ is a pseudo $M_\pi$-group, then $G$ is solvable.
\end{theorem}

We prove Theorems \ref{pseudo M_p} and \ref{pseudo M_pi} together.

\begin{proof}[Proof of Theorems \ref{pseudo M_p} and \ref{pseudo M_pi}]
Set ${\rm I} (G)$ to be ${\rm IBr} (G)$ if we are in Theorem \ref{pseudo M_p} and ${\rm I}_{\pi} (G)$ if we are in Theorem \ref{pseudo M_pi}.  Let $N$ be the solvable residual of $G$ and our goal is to prove that $N = 1$.
    %If $N=1$, then $G=G/1=G/N$ is solvable.
Thus, we assume that $N > 1$, and we choose $\varphi\in {\rm I} (G)$ of minimal degree with $N \nsubseteq \ker\varphi$.

If $S(\varphi) = \{ \varphi\in {\rm I}(G)\}$, then $H = G$ and $\varphi = \psi^{G}$ with $\psi \in {\rm I} (H)$, then we have $\varphi = \psi$ and by the hypothesis it follows that $\varphi$ is linear.  Thus, $G/\ker\varphi$ is cyclic, and so, $N \leq \ker\varphi$, a contradiction.

Hence, we assume $|S(\varphi)| > 1$.  Hence, we may assume that $H$ is a proper subgroup of $G$ and $\varphi = \psi^{G}$ with $\psi\in {\rm I} (H)$.  It follows that all nonprincipal irreducible constituents of $\theta = (1_{H^o})^{G}$ have degree strictly less than $\varphi (1) = |G:H| \psi(1)$.  Therefore, by the choice of $\varphi$, we have $N \subseteq \ker\theta \subseteq H$.

Now, let $\alpha$ be an irreducible constituent of $\varphi_{N}$.  Since $N \nsubseteq \ker\varphi$ we have $\alpha \neq 1_{N^o}$, and because $N = N'$, we deduce that $\alpha (1) > 1$.  By the hypothesis, we may choose $H < G$ and $\psi\in {\rm I} (H)$ with $\psi^{G} = \varphi$ and $\alpha(1) \nmid \psi(1)$.  In light of the previous paragraph, we have $N \subseteq H$.

Since $\psi^{G} = \varphi$, the irreducible constituents of $\psi_{N}$ are all $G$-conjugate to $\alpha$, and so, they all have degree $\alpha (1)$.  Therefore, $\alpha(1) \mid \psi(1)$ and this is a contradiction.  Thus, $N = 1$ and $G = G/1$ is solvable.
\end{proof}

We now turn to the analog of Okuyama's theorem that $M$-groups are $M_p$-groups for every prime $p$.  Notice that Theorem \ref{thm01} is a special case of Theorem \ref{thm01-Mpi} since an pseudo-$M$-group is necessarily solvable taking $\pi = p'$.

\begin{theorem}\label{thm01-Mpi}
If $G$ is a pseudo-$M$-group, then $G$ is a pseudo-$M_\pi$-group for every set of primes $\pi$.
\end{theorem}

\begin{proof}%[Proof of Theorem \ref{thm01-Mpi}]
Let $\varphi \in {\rm I}_{\pi} (G)$.  Suppose that there exists a subgroup $H \leq G$ with a $\pi$-partial character $\theta \in {\rm I}_\pi (H)$ so that $\theta^{G} = \varphi$.  Since $H$ is solvable, it follows that there exists an ordinary character $\psi \in {\rm Irr} (H)$ such that the restriction $\psi^o$ of $\psi$ to the set $H^o$ of $\pi$--elements of $H$ is $\theta$.  Thus,
$$\varphi = \theta^{G} = (\psi^o)^{G} = (\psi^{G})^o$$
and so, $\psi^{G}$ is an irreducible character of $G$.  Since $G$ is a pseudo-$M$-group, we conclude that $\psi^{G}$ is a pseudo-monomial character.  Suppose $\mu^{G} = \psi^{G}$ where $\mu \in {\rm Irr} (M)$ for a subgroup $M \le G$, it follows
$$(\mu^o)^{G} = (\mu^{G})^o = (\psi^{G})^o = \varphi,$$
and so, $\mu^o$ is an irreducible $\pi$-partial character.  Since $\psi^{G}$ is a pseudo-monomial character, it follows that ${\rm gcd} \{ \mu(1) \mid \mu\in D(\psi^{G}) \} = 1$, and thus, ${\rm gcd} \{ \theta (1) \mid \theta \in S(\varphi) \} = 1$.  We deduce that $\varphi$ is a pseudo-monomial $\pi$-partial character.  Therefore, $G$ is a pseudo-$M_\pi$-group for every set of primes $\pi$.
\end{proof}

This next theorem includes Theorem \ref{thm02} (2) as a special case since $G$ is solvable taking $\pi = p'$.  We obtain Theorem \ref{thm02} (1) by realizing that we may view $G$ as a pseudo-$M_p$-group for a prime $p$ that does not divide $|G|$ and apply Theorem \ref{thm02} (2).

\begin{theorem}\label{thm02-pi}
If $G$ is a pseudo-$M_{\pi}$-group for a set of primes $\pi$ and $N$ is a normal Hall subgroup of $G$, then $N$ is also a pseudo-$M_{\pi}$-group.
\end{theorem}

\begin{proof}%[Proof of Theorem \ref{thm02}]
Observe that $O_{\pi'} (N) = N \cap O_{\pi'} (G)$ and $N/(N \cap O_{\pi'} (G)) \cong NO_{\pi'} (G)/O_{\pi'} (G)$.  Thus, the $\pi$-partial characters of $N$ correspond to the $\pi$-partial characters of $NO_{\pi'} (G)$.  In particular, we may assume that $O_{\pi'} (G) = 1$.

Let $\theta$ be an irreducible $\pi$-partial character of $N$.  Then there exists $\varphi \in {\rm I}_\pi (G)$ so that $\theta$ is an irreducible constituent of $\varphi_{N}$ by Clifford's theorem (reference here Lemma 5.8 of \cite{solv text}).  Thus, we can find a subgroup $H$ of $G$ and a $\pi$-partial character $\psi \in {\rm I}_\pi (H)$ so that $\psi^{G} = \varphi$.  Observe that $(\psi^{NH})^{G} = \psi^{G}$, and so, $\psi^{NH} \in {\rm I}_\pi (NH)$.  Also,
$$\varphi (1) = \psi^{G} (1) = |G: H| \psi (1) = |G: NH| \psi^{NH} (1) = |G:NH| |NH:H| \psi (1).$$

Let $\tau$ be an irreducible constituent of $\psi_{N\cap H}$.  We may assume that $\theta$ is an irreducible constituent of $\tau^{N}$.  Since $\theta$ is an irreducible constituent of $\varphi_{N}$, it follows by
%\cite[Theorem 8.22]{Navarro1998}
that $\varphi(1)/\theta(1)$ divides $|G:N|$.  We have proved in Theorem \ref{pseudo M_pi} that pseudo-$M_{\pi}$-groups are solvable.
%(We have proved in \cite{CL} that pseudo-$M_{\pi}$-groups are solvable.)

Now, $N$ is a normal Hall $\rho$-subgroup of $G$, where $\rho$ may be considered as the set of prime divisors of $|N|$.  Notice that $N \cap H$ is a subgroup of $N$ and $N \cap H \lhd H$.  Since $(|N|, |G/N|)=1$, it follows that $(|N\cap H|, |H/N\cap H|) = 1$ because $H/(N\cap H)\cong HN/N\leq G/N$.  And thus, $N \cap H$ is a normal Hall $\rho$-subgroup of $H$.

Therefore, $\theta(1) = (\varphi(1))_{\rho}$, where $(\varphi(1))_{\rho}$ is the $\rho$-part of $\varphi (1)$, and $\tau (1) = (\psi (1))_{\rho}$ since $\psi (1)/\tau (1)$ divides $|H : (N \cap H)|=|HN : N|$. Thus, we have
$$\theta (1) = (\varphi (1))_{\rho} = (\psi^{NH} (1))_{\rho} = |N: N\cap H| (\psi (1))_{\rho} = |N: N\cap H| \tau (1) = \tau^{N} (1).$$
Since $\theta$ is an irreducible constituent of $\tau^{N}$, it follows that $\theta = \tau^{N}$ and $\tau (1)$ divides $\psi (1)$.  Therefore, we conclude by $G$ being a pseudo-$M_{\pi}$-group, that $N$ is also a pseudo-$M_{\pi}$-group.
\end{proof}

As we mentioned above, most of the work in Gunter's dissertation \cite{Gunter} is showing that certain pseudo $M$-groups have to be $M$-groups.  We want to close this paper with one result of this flavor.    Dade and Isaacs have shown when $p$ is an odd prime that monomial characters of $p$-power degree have the stronger property that every primitive character that induces the character will be linear, and then Isaacs uses this property to show that monomial $\{p\}$-partial characters also have the property that any primitive character that induces them will be linear (see note at the end of \cite{dade} and Theorems 10.1 and 10.2 of \cite{nato}).  We will use similar ideas to show that any pseudo-monomial character of $p$-power degree will be monomial and any pseudo-monomial $\{ p\}$-special character will be monomial.

\begin{lemma} \label{one}
Let $q$ be a prime and let $G$ be a group.
\begin{enumerate}
\item If $\chi \in \Irr (G)$ has $q$-power degree and is pseudo-monomial, then $\chi$ is monomial.
\item If $p$ is a prime and $\phi \in \IBr (G)$ has $q$-power degree and is pseudo-monomial, then $\phi$ is monomial.
\item If $\pi$ be a set of primes, $G$ is a $\pi$-separable group, $\phi \in {\rm I}_\pi (G)$ has $q$-power degree and is pseudo-monomial, then $\phi$ is monomial.
\end{enumerate}
\end{lemma}

\begin{proof}%[Proof of Lemmas \ref{one}, \ref{two}, and \ref{three}]
Since $\chi$ or $\phi$ is pseudo-monomial, we know that the greatest common divisor of the degrees of the primitive (Brauer, $\pi$-partial) characters inducing $\chi$ or $\phi$ is $1$.  However, all of the primitive (Brauer, $\pi$-partial) characters inducing $\chi$ or $\phi$ have $p$-power degree, so the only way that their greatest common divisor can be $1$ is if one of them has degree $1$.  This implies that $\chi$ or $\phi$ is induced by a linear (Brauer, $\pi$-partial) character.
\end{proof}

We close with showing that an $M_{\{ p\}}$-group must be an $M$-group.  Notice that this is a group where the $\{ p \}$-partial characters are monomial, not the $p$-Brauer characters.  The $p$-Brauer characters are actually the partial characters for the complement of $p$.  When discussing the group with the $p$-Brauer characters monomial, we denote as an $M_p$-group and this is the same as an $M_{\{p\}'}$-group.

\begin{theorem}
Let $p$ be an odd prime.  If $G$ is a pseudo $M_{\{ p \}}$-group, then $G$ is an $M$-group.
\end{theorem}

\begin{proof}
We show that every irreducible $\{ p \}$-partial character is monomial.  Let $\phi$ be an irreducible $\{p\}$-partial character of $G$.  Let $\chi$ be the ${\rm B}_{\{p\}} (G)$ that lifts $\phi$.  We know that every primitive character that induces $\chi$ is $p$-special and so, they all have $p$-power degree.  Next, observe that if $\sigma$ is a primitive $\{p\}$-partial character that induces $\phi$ and $\tau$ is the ${\rm B}_{\{p\}}$-lift for $\sigma$, then $\tau$ is primitive and induces $\chi$.  This implies that every primitive partial character inducing $\phi$ has $p$-power degree.  Since their greatest common divisor is $1$, one of them must be one.  Therefore, $\phi$ is monomial.
\end{proof}

One final note, in \cite{Isaacs19842}, Isaacs notes that he does not have any examples of pseudo $M$-groups that are not $M$-groups.  As far as we know, there are still no known examples of pseudo $M$-groups that are not $M$-groups.  Continuing in this vein, we do not have any pseudo $M_p$-groups that are not $M_p$-groups, nor any pseudo $M_\pi$-groups that are not $M_\pi$-groups.  We would be very interested to find examples of such groups or to prove that they cannot exist.

%%%%%%%%%%%%%%%%%%%%%%%%%%%%%%%%%%%%%%%%%%%%%%%%%%%%%%%%%%%%%%%%%%%%%%%%%%%%%%%%%%%%%%%%%%%%%%%%%%%%%%

\section*{Acknowledgments}

The first author thanks support of the program of Henan University of Technology (2024PYJH019),
the program of Foreign Experts of Henan Province (HNGD2024020),
and the Natural Science Foundation of Henan Province (252300421983).
%and the programs of Henan Province (HNGD2022044, HNGD2024020, 242300421384).

%%%%%%%%%%%%%%%%%%%%%%%%%%%%%%%%%%%%%%%%%%%%%%%%%%%%%%%%%%%%%%%%%%%%%%%%%%%%%%%%%%%%%%%%%%%%%%%%%%%%%%


\begin{thebibliography}{100}

\bibitem{super} X. Chen and J. Zheng, Super-Brauer characters and super-regular classes, {\it Monatsh Math.} {\bf 163} (2011), 15-23.

\bibitem{dade} E. C. Dade, Monomial characters and normal subgroups, {\it Math. Z.} {\bf 178} (1981), 401-420.

\bibitem{DIG}
K. Dastouri, A. Iranmanesh, and M. Ghasemi,
Monomial Isaacs $\pi$-partial characters and derived length,
{\it J. Algebra Appl.}, (2025) 2550108 (8 pages)

\bibitem{Gunter}
E. L. Gunter,
Pseudo-monomial characters and pseudo-M-groups,
Ph.D. Thesis, University of Wisconsin, 1987.

\bibitem{Isaacs1976}
I. M. Isaacs, {\it Character Theory of Finite Groups},
Academic Press, New York, 1976.

\bibitem{Isaacs1984}
I. M. Isaacs, Characters of $\pi$-separable groups,
{\it J. Algebra} {\bf 86} (1984), 98-128.


\bibitem{Isaacs19842}
I. M. Isaacs, Generalizations of Taketa's theorem on the solvability of $M$-groups,
{\it Proc. Amer. Math. Soc.} {\bf 91} (1984), 192-194.


\bibitem{Isaacs1994}
I. M. Isaacs, The $\pi$-character theory of solvable groups,
{\it J. Austral. Math. Soc. Ser. A} {\bf 57} (1994), 81-102.

\bibitem{nato} I.~M.~Isaacs, Characters and sets of primes for solvable groups in NATO Adv. Sci. Inst. Ser. C: Math. Phys. Sci., 471 Kluwer Academic Publishers Group, Dordrecht, 1995, 347--376.

\bibitem{solv text} %I.~M.~Isaacs, ``Characters of solvable groups,'' Grad. Stud. Math., 189
%American Mathematical Society, Providence, RI, 2018, ISBN: 978-1-4704-3485-4
I. M. Isaacs, {\it Characters of Solvable Groups}, American Mathematimathcal Society, Providence, R.I., 2018.


\bibitem{my pre}
M. L. Lewis, Kernels of Brauer characters and Isaacs' partial characters, preprint.

\bibitem{moreto}
A. Moret\'o, An answer to two questions of Brewster and Yeh on M-groups, {\it J. Algebra} {\bf 268} (2003), 366-370.

\bibitem{Navarro1998}
G. Navarro, {\it Characters and Blocks of Finite Groups},
Cambridge University Press, Cambridge, 1998.

\bibitem{Okuyama}
T. Okuyama, Module correspondence in finite groups,
{\it Hokkaido Math. J.} {\bf 10} (1981), 299-318.

\bibitem{vand} R. W. van der Waall, On the embedding of minimal non-M-groups, Nederl. Akad. Wetensch. Proc. Ser. A 77 {\it Indag. Math.} {\bf 36} (1974), 157-167.

\end{thebibliography}
\end{document}